\documentclass{amsart} 

\newtheorem{theorem}{Theorem}[section]
\newtheorem{lemma}[theorem]{Lemma}
\newtheorem{proposition}[theorem]{Proposition}

\theoremstyle{definition}
\newtheorem{definition}[theorem]{Definition}
\newtheorem{example}[theorem]{Example}

\newcommand{\Z}{\mathbb{Z}}

\usepackage{graphicx} 

\begin{document}


\title[A note on coverings of virtual knots]
{A note on coverings of virtual knots}

\author{Takuji NAKAMURA}
\address{Department of Engineering Science, 
Osaka Electro-Communication University,
Hatsu-cho 18-8, Neyagawa, Osaka 572-8530, Japan}
\email{n-takuji@osakac.ac.jp}

\author{Yasutaka NAKANISHI}
\address{Department of Mathematics, Kobe University, 
Rokkodai-cho 1-1, Nada-ku, Kobe 657-8501, Japan}
\email{nakanisi@math.kobe-u.ac.jp}

\author{Shin SATOH}
\address{Department of Mathematics, Kobe University, 
Rokkodai-cho 1-1, Nada-ku, Kobe 657-8501, Japan}
\email{shin@math.kobe-u.ac.jp}

\renewcommand{\thefootnote}{\fnsymbol{footnote}}
\footnote[0]{The first author is partially supported by 
JSPS Grants-in-Aid for Scientific Research (C), 
17K05265. 
The third author is partially supported by 
JSPS Grants-in-Aid for Scientific Research (C), 
16K05147.}


\renewcommand{\thefootnote}{\fnsymbol{footnote}}





\maketitle


\begin{abstract} 
For a virtual knot $K$ and an integer $r\geq 0$, 
the $r$-covering $K^{(r)}$ is defined 
by using the indices of chords on a Gauss diagram of $K$. 
In this paper, we prove that 
for any finite set of virtual knots 
$J_0,J_2,J_3,\dots,J_m$,  
there is a virtual knot $K$ such that 
$K^{(r)}=J_r$ $(r=0\mbox{ and }2\leq r\leq m)$, 
$K^{(1)}=K$, 
and otherwise $K^{(r)}=J_0$. 
\end{abstract}


\section{Introduction}\label{sec1} 

Odd crossings are first introduced for constructing 
a simple invariant called the odd writhe of a virtual knot by Kauffman \cite{Kau2}. 
By using odd crossings, 
Manturov defines a map 
from the set of virtual knots to itself 
by replacing the odd crossings with virtual crossings \cite{Man}. 

Later the notion of index is introduced 
in \cite{CG, Hen, ILL, ST}
which assigns an integer to each real crossing 
such that the parity of the index coincides with 
the original parity. 
The $n$-writhe is defined as 
a refinement of the odd writhe. 
Jeong defines an invariant called the zero polynomial of a virtual knot 
by using real crossings of index $0$ \cite{Jeo}. 
In fact, Im and Kim prove that 
the zero polynomial is coincident with 
the writhe polynomial of the virtual knot 
obtained by replacing the real crossings 
whose indices are non-zero 
with virtual crossings \cite{IK}. 
They also study the operation replacing 
the real crossings whose indices are not divisible 
by $r$ for a positive integer $r$. 
This operation is originally considered 
for flat virtual knots by Turaev \cite{Tur} 
where he calls the obtained knot the {\it $r$-covering}.

The writhe polynomial $W_K(t)$ of a virtual knot $K$ 
is the polynomial such that the coefficient of $t^n$ is equal to 
the $n$-writhe of $K$. 
A characterization of $W_K(t)$ is given as follows.

\begin{theorem}[{\cite{ST}}]\label{thm11}
For a Laurent polynomial $f(t)\in{\Z}[t,t^{-1}]$, 
the following are equivalent. 
\begin{itemize}
\item[(i)] 
There is a virtual knot $K$ with $W_K(t)=f(t)$. 
\item[(ii)] 
$f(1)=f'(1)=0$. 
\end{itemize}
\end{theorem}

For an integer $r\geq 0$, 
we denote by $K^{(r)}$ 
the $r$-covering of a virtual knot $K$. 
By definition, 
we have $K^{(1)}=K$ and 
$K^{(r)}=K^{(0)}$ 
for a sufficiently large $r$. 
The aim of this note is to prove the following.

\begin{theorem}\label{thm12}
Let $m\geq 1$ be an integer, $J_n$ $(0\leq n\leq m, n\ne 1)$  
$m$ virtual knots, 
and $f(t)$ a Laurent polynomial with $f(1)=f'(1)=0$. 
Then there is a virtual knot $K$ such that 
$$K^{(r)}=
\left\{
\begin{array}{ll}
J_0 & \mbox{for }r=0 \mbox{ and }r\geq m+1, \\ 
K & \mbox{for } r=1,  \\
J_r & \mbox{otherwise}, \\
\end{array}
\right.$$
and $W_K(t)=f(t)$. 
\end{theorem}

This paper is organized as follows. 
In Section~\ref{sec2}, 
we define the $r$-covering $K^{(r)}$ 
of a (long) virtual knot $K$. 
We also introduce an anklet of a chord 
in a Gauss diagram which will be used in the consecutive sections. 
In Sections~\ref{sec3} and \ref{sec4}, 
we study the $0$-covering $K^{(0)}$ 
and $r$-covering $K^{(r)}$ for $r\geq 2$ 
of a long virtual knot, 
respectively. 
In Section~\ref{sec5}, 
we review the writhe polynomial of a virtual knot, 
and prove Theorem~\ref{thm11}.


\section{Gauss diagrams}\label{sec2} 

A {\it circular} or {\it linear Gauss diagram}  is an oriented circle or line 
equipped with a finite number of oriented and signed chords 
spanning the circle or line, respectively. 
The {\it closure} of a linear Gauss diagram 
is the circular Gauss diagram 
obtained by taking the one-point compactification of the line.

Let $c$ be a chord of a Gauss diagram $G$ 
with sign $\varepsilon=\varepsilon(c)$. 
We give signs $-\varepsilon$ and $\varepsilon$ 
to the initial and terminal endpoints of $c$, respectively. 
We consider the case that $G$ is circular. 
The endpoints of  $c$ divide the circle into two arcs. 
Let $\alpha$ be the arc oriented 
from the initial endpoint of $c$ to the terminal. 
See Figure~\ref{fig01}. 
The {\it index} of $c$ is the sum of signs 
of endpoints of chords on $\alpha$, 
and denoted by ${\rm Ind}_G(c)$ 
(cf.~\cite{Che, Kau, ST}). 
In the case that $G$ is linear, 
the index of $c$ is defined 
as that of $c$ in the closure of $G$.

\begin{figure}[htb]
\begin{center}
\includegraphics[bb=0 0 81 61]{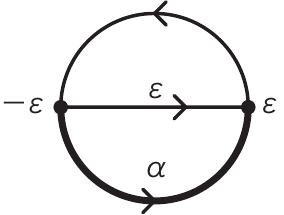}
\caption{The orientation and signs of a chord and its endpoints.}
\label{fig01}
\end{center}
\end{figure}

Let $G$ be a circular or linear Gauss diagram. 
For a positive integer $r$, 
we denote by $G^{(r)}$ 
the Gauss diagram obtained from $G$ 
by removing all the chords $c$ with 
${\rm Ind}_G(c)\not\equiv 0~({\rm mod}\ r)$ 
(cf.~\cite{Tur}). 
In particular, we have $G^{(1)}=G$. 
For $r=0$, 
we denote by $G^{(0)}$ the Gauss diagram 
obtained from $G$ by removing 
all the chords $c$ with ${\rm Ind}_G(c)\ne 0$. 
Since the number of chords of $G$ is finite, 
we have $G^{(r)}=G^{(0)}$ for sufficiently large $r$.

Two circular Gauss diagrams $G$ and $H$ are {\it equivalent}, 
denoted by $G\sim H$, 
if $G$ is related to $H$ 
by a finite sequence of {\it Reidemeister moves} I--III 
as shown in Figure~\ref{fig02}. 
A {\it virtual knot} is an equivalence class of circular Gauss diagrams  
up to this equivalence relation 
(cf.~\cite{GPV, Kau2}). 
Similarly, the equivalence relation among 
linear Gauss diagrams are defined, 
and an equivalence class is called 
a {\it long virtual knot}. 
The {\it trivial} (long) virtual knot 
is presented by a Gauss diagram with no chord.

\begin{figure}[htb]
\begin{center}
\includegraphics[bb=0 0 342 279]{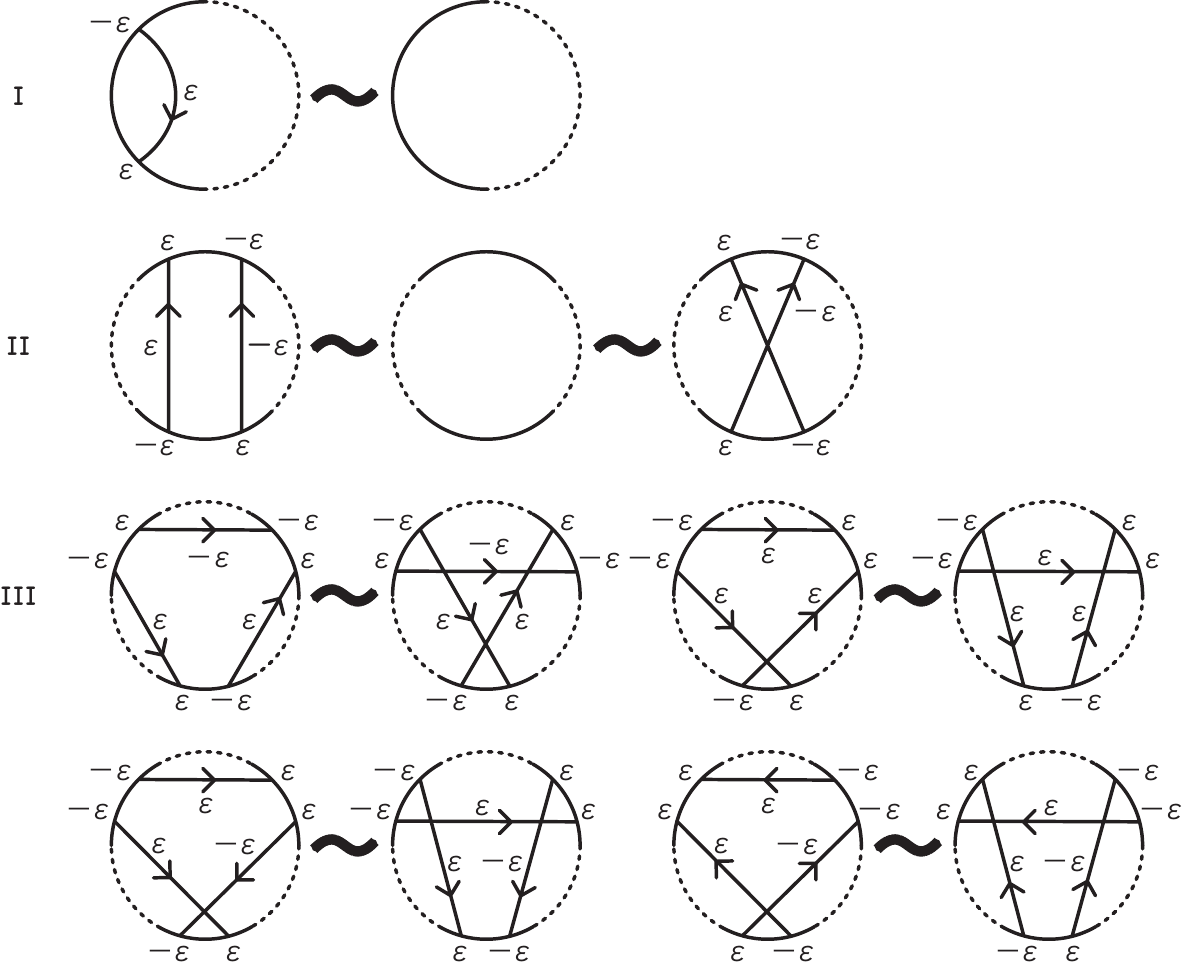}
\caption{Reidemeister moves on Gauss diagrams.}
\label{fig02}
\end{center}
\end{figure}

\begin{lemma}[cf.~\cite{IK, Tur}]\label{lem21}
Let $G$ and $H$ be circular or linear Gauss diagrams such that $G\sim H$. 
Then it holds that $G^{(r)}\sim H^{(r)}$ for any integer $r\geq 0$. 
\hfill$\Box$
\end{lemma}

Although only the case of a circular Gauss diagram is 
studied in \cite{IK, Tur}, 
Lemma~\ref{lem21} for a liner Gauss diagram 
can be proved similarly.

\begin{definition}\label{def22}
Let $K$ be a (long) virtual knot, and 
$r\geq 0$ an integer. 
The {\it $r$-covering} of $K$ is the (long) virtual knot 
presented by $G^{(r)}$ for some Gauss diagram $G$ of $K$. 
We denote it by $K^{(r)}$. 
\end{definition}

The well-definedness of $K^{(r)}$ follows 
from Lemma~\ref{lem21}. 
We have $K^{(1)}=K$ and 
$K^{(r)}=K^{(0)}$ for sufficiently large $r$. 

Let $c(G)$ denote the number of chords of a Gauss diagram $G$. 
The {\it real crossing number} of a (long) virtual knot $K$ 
is the minimal number of $c(G)$ 
for all Gauss diagrams $G$ of $K$, 
and denoted by ${\rm c}(K)$.

\begin{lemma}\label{lem23}
Let $K$ be a $($long$)$ virtual knot, and $r\geq 0$ an integer. 
Then it holds that 
${\rm c}(K^{(r)})\leq {\rm c}(K)$. 
In particular, ${\rm c}(K^{(r)})={\rm c}(K)$ holds 
if and only if $K^{(r)}=K$. 
\end{lemma}

\begin{proof} 
Let $G$ be a Gauss diagram of $K$ with 
$c(G)={\rm c}(K)$. 
Since $G^{(r)}$ is obtained from $G$ by removing some chords, 
we have 
$${\rm c}(K^{(r)})\leq c(G^{(r)})\leq c(G)={\rm c}(K).$$
In particular, if the equality holds, 
then $G^{(r)}=G$ and $K^{(r)}=K$. 
\end{proof}

Let $c_1,\dots,c_n$ be chords of a Gauss diagram $G$. 
We add several parallel chords near an endpoint of $c_i$ 
$(i=1,\dots,n)$ to obtain a Gauss diagram $G'$. 
Here, the orientations and signs of the added chords 
are chosen arbitrarily. 
See Figure~\ref{fig03}(i). 
The parallel chords added to $c_i$ are called {\it anklets} of $c_i$. 
We remark that the index of an anklet is equal to $\pm 1$.

\begin{figure}[htb]
\begin{center}
\includegraphics[bb=0 0 307 78]{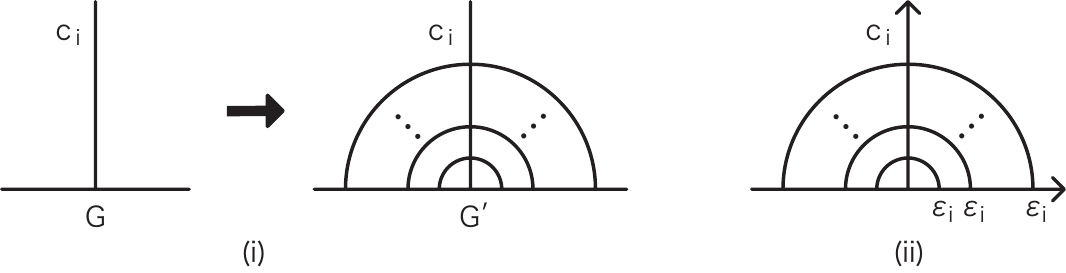}
\caption{Anklets.}
\label{fig03}
\end{center}
\end{figure}

\begin{lemma}
Let $c_1,\dots,c_n$ be chords of a Gauss diagram $G$. 
For any integers $a_1,\dots,a_n$, 
by adding several anklets to each $c_i$ near its initial endpoint suitably, 
we obtain a Gauss diagram $G'$ 
such that ${\rm Ind}_{G'}(c_i)=a_i$ for 
any $i=1,\dots,n$, 
and ${\rm Ind}_{G'}(c)={\rm Ind}_{G}(c)$ 
for any $c\ne c_1,\dots,c_n$. 

\end{lemma}

\begin{proof}
Put $d_i=a_i-{\rm Ind}_G(c_i)$ for $i=1,\dots,n$. 
We add $|d_i|$ anklets to $c_i$ near its initial endpoint 
such that the signs of right endpoints of the anklets  
are equal to $\varepsilon_i$, 
where $\varepsilon_i$ is the sign of $d_i$. 
See Figure~\ref{fig03}(ii). 

Let $G'$ be the obtained Gauss diagram. 
Then we have 
$${\rm Ind}_{G'}(c_i)=
{\rm Ind}_G(c_i)+\varepsilon_i|d_i|=
{\rm Ind}_G(c_i)+d_i=a_i.$$
Furthermore 
the index of a chord other than $c_1,\dots,c_n$ 
does not change. 
\end{proof}


\section{The $0$-covering $K^{(0)}$}\label{sec3}

For an integer $n\geq 2$, 
we define a map 
$f_n:\{2,3,\dots,n\}\rightarrow{\Z}$ 
which satisfies 
$$\sum_{{r\leq i\leq n}\atop{i\equiv 0\ ({\rm mod}\ r)}} f_n(i)=-1$$
for any integer $r$ with $2\leq r\leq n$. 
The map $f_n$ exists uniquely. 
Put $$P_n=\{i \mid 2\leq i\leq n, f_n(i)\ne 0\}.$$

\begin{example}
For $n=10$, we have 
$$\left\{
\begin{array}{lrl}
f_{10}(2)+f_{10}(4)+f_{10}(6)+f_{10}(8)+f_{10}(10)&=-1 & \mbox{for } r=2, \\
f_{10}(3)+f_{10}(6)+f_{10}(9)&=-1 & \mbox{for } r=3, \\
f_{10}(4)+f_{10}(8)&=-1 & \mbox{for } r=4, \\
f_{10}(5)+f_{10}(10)&=-1 & \mbox{for }r=5, \\
f_{10}(6)&=-1 & \mbox{for } r=6, \\
f_{10}(7)&=-1 & \mbox{for } r=7, \\
f_{10}(8)&=-1 & \mbox{for } r=8, \\
f_{10}(9)&=-1 & \mbox{for } r=9, \\
f_{10}(10)&=-1 & \mbox{for } r=10. 
\end{array}\right.$$
Therefore we have 
$$f_{10}(i)=
\left\{
\begin{array}{rl}
2 & (i=2), \\
1 & (i=3), \\
0 & (i=4,5), \\
-1 & (6\leq i\leq 10) 
\end{array}\right.
\mbox{ and } 
P_{10}=\{2,3,6,7,8,9,10\}.$$
\end{example}

\begin{theorem}\label{thm32}
For any integer $n\geq 1$ and long virtual knot $J$, 
there is a long virtual knot $K$ such that 
$$K^{(r)}=
\left\{
\begin{array}{ll}
J & \mbox{for }r=0 \mbox{ and }r\geq n+1, \\
K & \mbox{for } r=1, \ and \\
O & \mbox{otherwise}. \\
\end{array}
\right.$$
Here, $O$ denotes the trivial long virtual knot. 
\end{theorem}

\begin{proof}
Let $H$ be a linear Gauss diagram of $J$. 
We construct a linear Gauss diagram $G$ as follows: 
First, we replace each chord $c$ of $H$ by 
$1+\sum_{i\in P_n}|f_n(i)|$ parallel chords labeled 
$c_0$ and $c_{ij}$ 
for $i\in P_n$ and $1\leq j\leq |f_n(i)|$. 
The orientations of $c_0$ and $c_{ij}$'s are 
the same as that of $c$. 
The signs of them are given such that 
\begin{itemize}
\item[(i)] 
$\varepsilon(c_0)=\varepsilon(c)$, and 
\item[(ii)] 
$\varepsilon(c_{ij})=\varepsilon(c)\delta_i$ for $i\in P_n$, 
where $\delta_i$ is the sign of $f_n(i)\ne 0$. 
\end{itemize}
Next, 
we add several anklets to each of $c_0$ and $c_{ij}$'s 
such that 
\begin{itemize}
\item[(iii)] 
${\rm Ind}_G(c_0)=0$, and 
\item[(iv)] 
${\rm Ind}_G(c_{ij})=i$ for $i\in P_n$. 
\end{itemize}
The obtained Gauss diagram is denoted by $G$.

Figure~\ref{fig04} shows the case $n=10$ 
replacing each chord $c$ of $H$ with nine chords 
$c_0,c_{21},c_{22}, c_{31},\dots,c_{10,1}$ and several anklets. 
The boxed numbers of the chords indicate their indices.

\begin{figure}[htb]
\begin{center}
\includegraphics[bb=0 0 352 69]{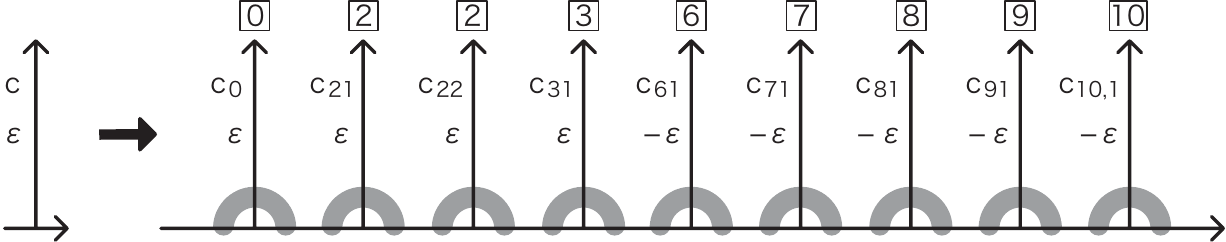}
\caption{The case $n=10$.}
\label{fig04}
\end{center}
\end{figure}

By the conditions (i)--(iv), 
we have $G^{(0)}=H$; 
in fact, we remove the chords whose indices are non-zero from $G$ 
to obtain $G^{(0)}$. 
Similarly, we have $G^{(r)}=H$ for any $r\geq n+1$. 
Therefore 
it holds that $K^{(0)}=K^{(r)}=J$ 
for $r\geq n+1$.

For $2\leq r\leq n$, 
$G^{(r)}$ is obtained from $G$ by removing the chords 
whose indices are not divisible by $r$. 
In particular, all the anklets are removed. 
Among the chords $c_0$ and $c_{ij}$'s, 
the sum of signs of chords 
whose indices are divisible by $r$ is equal to 
$$\varepsilon(c)+\varepsilon(c)
\sum_{{r\leq i\leq n}\atop{i\equiv 0\ ({\rm mod}\ r)}} 
\delta_i|f_n(i)|=
\varepsilon(c)+\varepsilon(c)
\sum_{{r\leq i\leq n}\atop{i\equiv 0\ ({\rm mod}\ r)}} 
f_n(i)=0.$$ 
Therefore all the chords of $G^{(r)}$ can be canceled 
by Reidemesiter moves II so that 
$K^{(r)}$ is the trivial long virtual knot. 
\end{proof}


\section{The $r$-covering $K^{(r)}$ for $r\geq 2$}\label{sec4} 

For an integer $n\geq 2$, we define a map 
$g_n:\{2,3,\dots,n\}\rightarrow{\Z}$ 
which satisfies 
$$g_n(n)=1 \mbox{ and } 
\sum_{{r\leq i\leq n}\atop{i\equiv 0\ ({\rm mod}\ r)}} g_n(i)=0$$
for any integer $r$ with $2\leq r< n$. 
The map $g_n$ exists uniquely. 
Put $$Q_n=\{i\mid 2\leq i\leq n, g_n(i)\ne 0\}.$$

\begin{example}
(i) 
For $n=10$, it holds that 
$$\left\{
\begin{array}{lrl}
g_{10}(2)+g_{10}(4)+g_{10}(6)+g_{10}(8)+g_{10}(10)&=0 & \mbox{for } r=2, \\
g_{10}(3)+g_{10}(6)+g_{10}(9)&=0 & \mbox{for } r=3, \\
g_{10}(4)+g_{10}(8)&=0 & \mbox{for } r=4, \\
g_{10}(5)+g_{10}(10)&=0 & \mbox{for }r=5, \\
g_{10}(6)&=0 & \mbox{for } r=6, \\
g_{10}(7)&=0 & \mbox{for } r=7, \\
g_{10}(8)&=0 & \mbox{for } r=8, \\
g_{10}(9)&=0 & \mbox{for } r=9, \\
g_{10}(10)&=1 & \mbox{for } r=10. 
\end{array}\right.$$
Therefore we have 
$$g_{10}(i)=
\left\{
\begin{array}{rl}
1 & (i=10), \\
-1 & (i=2,5), \\
0 & ({\rm otherwise}), \\
\end{array}\right.
\mbox{ and } 
Q_{10}=\{2,5,10\}.$$

(ii) For $n=12$, we have 
$$g_{12}(i)=
\left\{
\begin{array}{rl}
1 & (i=2,12), \\
-1 & (i=4,6), \\
0 & ({\rm otherwise}), \\
\end{array}\right.
\mbox{ and } 
Q_{12}=\{2,4,6,12\}.$$
\end{example}

\begin{theorem}\label{thm42}
For any integer $n\geq 2$ and long virtual knot $J$, 
there is a long virtual knot $K$ such that 
$$K^{(r)}=
\left\{
\begin{array}{ll}
J & \mbox{for }r=n, \\
K & \mbox{for } r=1, \ and \\
O & \mbox{otherwise}. \\
\end{array}
\right.$$
\end{theorem}

\begin{proof}
The proof is similar to that of Theorem~\ref{thm32} 
by using $g_n$ instead of $f_n$. 
Let $H$ be a linear Gauss diagram of $J$. 
We construct a linear Gauss diagram $G$ of $K$ 
as follows: 
First, we replace each chord $c$ of $H$ 
by $\sum_{i\in Q_n} |g_n(i)|$ parallel chords 
labeled $c_i$ for $i\in Q_n$. 
The orientations of $c_i$'s are the same as that of $c$. 
The signs of them are given such that 
$\varepsilon(c_i)=\varepsilon(c)\delta_i$, 
where $\delta_i$ is the sign of $g_n(i)\ne 0$. 
Next, we add several anklets to each of $c_i$'s 
such that 
${\rm Ind}_G(c_i)=i$ for $i\in Q_n$. 
The obtained Gauss diagram is denoted by $G$.

In the left of Figure~\ref{fig05}, 
we shows the case $n=10$ 
replacing each chord $c$ of $H$ with three chords 
$c_2,c_5,c_{10}$ and several anklets. 
In the right figure, 
the case of $n=12$ is given.

\begin{figure}[htb]
\begin{center}
\includegraphics[bb=0 0 325 69]{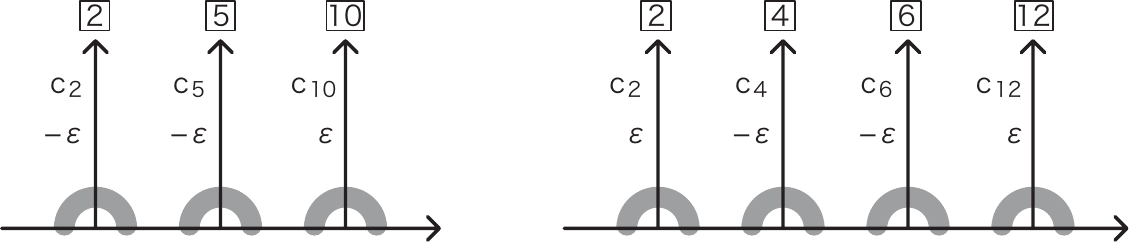}
\caption{The cases $n=10$ and $12$.}
\label{fig05}
\end{center}
\end{figure}

Since any chord $c$ of $G$ satisfies 
$1\leq |{\rm Ind}_G(c)|\leq n$, 
we obtain $G^{(0)}$ and $G^{(r)}$ for $r\geq n+1$ 
by removing all the chords from $G$.

For $2\leq r<n$, 
$G^{(r)}$ is obtained from $G$ by 
removing the chords 
whose indices are not divisible by $r$. 
Among the chords $c_i$'s, 
the sum of signs of chords whose indices are divisible by $r$ is 
equal to 
$$\varepsilon(c)
\sum_{{r\leq i\leq n}\atop{i\equiv 0\ ({\rm mod}\ r)}} \delta_i |g_n(i)|=
\varepsilon(c)
\sum_{{r\leq i\leq n}\atop{i\equiv 0\ ({\rm mod}\ r)}} g_n(i) =0.$$
Therefore all the chords of $G^{(r)}$ can be 
canceled by Reidemeister moves II so that $K^{(r)}$ 
is the trivial long virtual knot. 

Finally, for $r=n$, we have $G^{(n)}=H$ by definition immediately. 
\end{proof}

We see that $g_n$ 
is coincident with a famous function as follows.

\begin{proposition}\label{prop43}
Let $\mu$ be the M\"obius function. 
Then we have 
$$g_n(i)=
\left\{
\begin{array}{cl}
\mu(\frac{n}{i}) & \mbox{if $n$ is divisible by $i$, and}\\
0 & \mbox{otherwise}. 
\end{array}\right.$$
\end{proposition}

\begin{proof}
Let $h_n(i)$ be the right hand side of the equation in the proposition. 
Since $h_n(n)=\mu(1)=1=g_n(n)$, 
it is sufficient to prove that 
$$\sum_{{r\leq i\leq n}\atop{i\equiv 0\ ({\rm mod}\ r)}} h_n(i)=0$$
for any integer $r$ with $2\leq r< n$. 

Assume that $n$ is not divisible by $r$. 
Then $n$ is not divisible by any $i$ 
such that $r\leq i\leq n$ and $i\equiv 0$ (mod~$r$). 
Therefore we have 
$$\sum_{{r\leq i\leq n}\atop{i\equiv 0\ ({\rm mod}\ r)}} h_n(i)=
\sum_{{r\leq i\leq n}\atop{i\equiv 0\ ({\rm mod}\ r)}} 0=0.$$

Assume that $n$ is divisible by $r$. 
By the property of the M\"obius function, it holds that 
$$\sum_{{r\leq i\leq n}\atop{i\equiv 0\ ({\rm mod}\ r)}} h_n(i)=
\sum_{r|i|n} \mu(\frac{n}{i})
=\sum_{d|\frac{n}{r}}\mu(d)=0.$$

Therefore we have $g_n=h_n$. 
\end{proof}


\section{The writhe polynomial}\label{sec5} 

For an integer $n\ne 0$ 
and a sign $\varepsilon=\pm 1$, 
the {\it $(n,\varepsilon)$-snail} is a linear Gauss diagram 
consisting of a chord $c$ with $\varepsilon(c)=\varepsilon$ 
and $|n|$ anklets 
such that the indices of $c$ and each anklet 
are equal to $n$ and $1$, respectively. 
See Figure~\ref{fig06}.

\begin{figure}[htb]
\begin{center}
\includegraphics[bb=0 0 343 165]{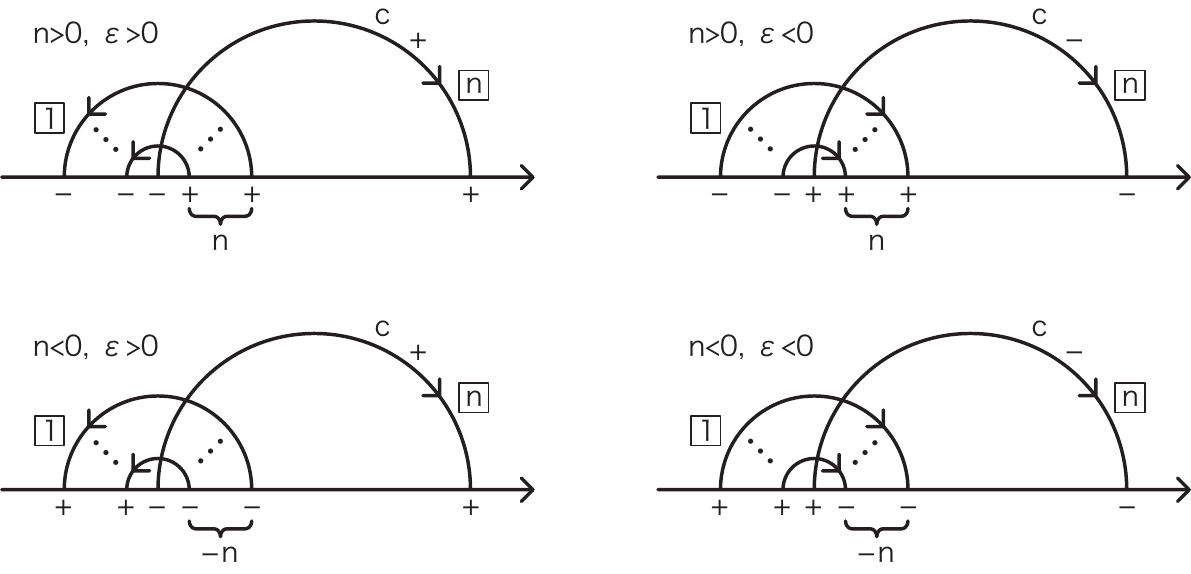}
\caption{The snail $S(n,\varepsilon)$.}
\label{fig06}
\end{center}
\end{figure}

Let $G$ be a Gauss diagram of a (long) virtual knot $K$. 
For an integer $n\ne 0$, 
we denote by $w_n(G)$ the sum of signs of all chords $c$ of $G$ 
with ${\rm Ind}_G(c)=n$. 
Then $w_n(G)$ does not depend on a particular choice of 
$G$ of $K$; 
that is, $w_n(G)$ is an invariant of $K$. 
In \cite{ST}, the proof is given for a virtual knot, 
and the case of a long virtual knot is similarly proved. 
It is called the {\it $n$-writhe} of $K$ 
and denoted by $w_n(K)$. 
The {\it writhe polynomial} of $K$ 
is defined by 
$$W_K(t)=\sum_{n\ne 0}w_n(K)t^n-
\sum_{n\ne 0}w_n(K)\in{\Z}[t,t^{-1}].$$
This invariant was introduced in several papers 
\cite{CG, Kau, ST} independently.

\begin{theorem}\label{thm51}
Let $J$ be a long virtual knot, 
and $f(t)\in{\Z}[t,t^{-1}]$ a Laurent polynomial with 
$f(1)=f'(1)=0$. 
Then there is a long virtual knot $K$ such that 
\begin{itemize}
\item[(i)] 
$K^{(r)}=J^{(r)}$ for any integer $r=0$ and $r\geq 2$, and 
\item[(ii)] 
$W_K(t)=f(t)$. 
\end{itemize}
\end{theorem}

\begin{proof}
Put $g(t)=f(t)-W_J(t)=\sum_{n\in{\Z}}a_nt^n$. 
By Theorem~\ref{thm11}, we have $g(1)=g'(1)=0$. 
Therefore it holds that 
$a_0=\sum_{n\ne 0,1}(n-1)a_n$ and 
$a_1=-\sum_{n\ne 0,1} na_n$.

Let $H$ be a linear Gauss diagram of $J$. 
We construct a linear Gauss diagram $G$ 
by juxtaposing $H$ and $|a_n|$ copies of $S(n,\varepsilon_n)$ 
for every integer $n$ with $n\ne 0,1$ and $a_n\ne 0$. 
Here, $\varepsilon_n$ is the sign of $a_n\ne 0$.

Let $K$ be the long virtual knot presented by $G$. 
The contribution of each snail $S(n,\varepsilon_n)$ 
to the writhe polynomial $W_K(t)$ is equal to 
$\varepsilon_nt^n-\varepsilon_nnt+\varepsilon_n(n-1)$. 
Therefore it holds that 

\begin{eqnarray*}
W_K(t)&=&
W_J(t)+
\sum_{n\ne 0,1} |a_n|\bigr(\varepsilon_nt^n-\varepsilon_nnt+\varepsilon_n(n-1)\bigr)\\
&=&
W_J(t)+\sum_{n\ne 0,1} a_n\bigr(t^n-nt+(n-1)\bigr)\\
&=&
W_J(t)+g(t)=f(t). 
\end{eqnarray*}

By definition, 
$S(n,\varepsilon_n)^{(r)}$ has the only chord $c$  
if $n$ is divisible by $r$. 
Otherwise it has no chord. 
Therefore $G^{(r)}$ is equivalent to $H^{(r)}$, 
and hence $K^{(r)}=J^{(r)}$ 
for any integer $r=0$ and $r\geq 2$. 
\end{proof}

\begin{theorem}\label{thm52}
Let $m\geq 1$ be an integer, $J_n$ $(0\leq n\leq m, n\ne 1)$  
$m$ long virtual knots, 
and $f(t)$ a Laurent polynomial with $f(1)=f'(1)=0$. 
Then there is a long virtual knot $K$ such that 
$$K^{(r)}=
\left\{
\begin{array}{ll}
J_0 & \mbox{for }r=0 \mbox{ and }r\geq m+1, \\ 
K & \mbox{for } r=1,  \\
J_r & \mbox{otherwise}, \\
\end{array}
\right.$$
and $W_K(t)=f(t)$. 
\end{theorem}

\begin{proof}
Let $K_0$ be a long virtual knot 
obtained by applying Theorem~\ref{thm32} 
to the pair of $m$ and $J_0$. 
Let $K_n$ be a long virtual knot 
obtained by applying Theorem~ \ref{thm42} 
to each pair of $n$ and $J_n$ 
$(2\leq n\leq m)$. 
We juxtapose $K_0,K_2,\dots,K_m$ 
to have a long virtual knot $K'$. 
Let $K$ be a long virtual knot 
obtained by applying Theorem~\ref{thm51} 
to the pair of $K'$ and $f(t)$. 
Then we see that 
$K$ is a desired long virtual knot. 
\end{proof}

\begin{proof}[Proof of {\rm Theorem~\ref{thm12}}] 
Let $J_n^{\circ}$ be a long virtual knot 
whose closure is $J_n$ 
$(0\leq n\leq m, n\ne 1)$. 
Let $K^{\circ}$ be a long virtual knot 
obtained by applying Theorem~\ref{thm52}. 
Then we see that the closure of $K^{\circ}$ 
is a desired virtual knot. 
\end{proof}


\end{document}